\documentclass[12pt, leqno]{amsart}

\usepackage{amsfonts,amsthm,amsmath,xcolor,comment}
\numberwithin{equation}{section}

\usepackage{amssymb}

\theoremstyle{plain}
\newtheorem{thm}{Theorem}[section]

\newtheorem{lem}{Lemma}[section]

\theoremstyle{definition}
\newtheorem{df}{Definition}[section]
\newtheorem{rem}{Remark}[section]

\newcommand{\ZZ}{\mathbb{Z}}

\newcommand{\CC}{\mathbb{C}}
\newcommand{\Lap}{\mathrm{Lap}}

\DeclareMathOperator{\Cay}{Cay}
\DeclareMathOperator{\Kf}{Kf}

\begin{document}


\title{On the average hitting times of weighted Cayley graphs}
\author{Yuuho Tanaka}

\address{Faculty of Science and Engineering, Waseda University, Tokyo 169-8555, Japan}
\curraddr{Faculty of Science and Technology, Oita University, Oita, Oita 870-1192, Japan}
\email{tanaka-yuuho@oita-u.ac.jp}
\date{}
\maketitle

\begin{abstract}
In this paper, we give exact formulas for the average hitting times of random walks from one vertex to any other vertex on certain weighted Cayley graphs.
\end{abstract}
\noindent
Keywords: random walk, hitting time, weighted graph, Cayley graph, Kirchhoff index

\section{Introduction}

Let $G=(V,E)$ be a graph.
A {\it weight} on $G$ is a function $w:V\times V\to[0,+\infty )$ such that for $u,v\in V$,
\[
\begin{cases}
w((u,v))>0 &\text{if $(u,v)\in E$,}\\
w((u,v))=0 &\text{otherwise.}
\end{cases}
\]
In other words, $w$ assigns a positive number $w((u,v))$ to each edge $(u,v)\in E$.
Such a graph $G$ is called a {\it weighted graph}.
Every graph $G=(V,E)$ can be given the {\it trivial weight} $w((u,v))=1$ if $(u,v)\in E$ and $w((u,v))=0$ otherwise. 
In this case, we refer to $G$ as an \textit{unweighted graph}.

A random walk on a graph $G$ is a discrete stochastic model
such that a random walker on a vertex $u\in V$ moves to an adjacent vertex $v$ at the next step with probability $\frac{w((u,v))}{W(u)}$, where $W(u)=\sum_{x\in V}w((u,x))$.
A random walk on an unweighted graph (namely $w((u,v))/W(u)=1/\deg(u)$, where $\deg(u)=|\{v\in V \mid (u,v)\in E\}|$) is called a {\it simple random walk}.
The number of steps required until the random walker starting from
a vertex $u$ of $G$ will first arrive at a vertex $v$ of $G$ is called
the {\it hitting time} from $u$ to $v$ of the random walk on $G$.
The {\it average hitting time} (AHT) from $u$ to $v$ on $G$, which is denoted by $h(G;u,v)$, means the expected value of the hitting time
from $u$ to $v$ of random walks on $G$.

Using an elementary method, Doi et al. \cite{HT} gave the exact formula for the AHTs of simple random walks on the unweighted Cayley graph $\Cay(\ZZ_N,\{\pm1,\pm2\})$.
Furthermore, Tanaka \cite{YT2024} gave the exact formula for the AHTs of simple random walks on the unweighted Cayley graph $\Cay(\ZZ_N,\{+1,+2\})$.
The purpose of the present paper is to derive exact formulas for the AHTs of random walks on certain weighted Cayley graphs by using matrix equations and matrix decompositions.
Chang et al.~\cite{XHS} obtained an explicit formula for the AHTs of random walks on weighted undirected graphs by using the formula for the number of spanning trees of the graphs.
Therefore, the aim of this paper is to obtain exact formulas for the AHTs of random walks on the weighted Cayley graphs $\Cay(\ZZ_{2n},\{\pm1\})$ and $\Cay(\ZZ_N,\{+1,+2\})$. 

For undirected Cayley graphs, in the method using matrix equations, the entries of the inverse of the coefficient matrix can be expressed using the sequence of numbers appearing in the formula for the number of spanning trees~\cite{HT, Kleitman}.
Therefore, we hypothesize that a similar property holds for weighted Cayley graphs.
For a weighted graph, we can obtain the weighted spanning tree enumerator from the weighted matrix-tree theorem~\cite{Kirchhoff, KS}.
Thus, we can deduce the entries of the inverse of the coefficient matrix in this case.

Our proof is concise and fully combinatorial.
In particular, it does not require any spectral graph theoretical arguments.
For the weighted directed Cayley graph $\Cay(\ZZ_N,\{+1,+2\})$, 
we develop a novel method by combining the Laplacian matrix approach $LU$ decomposition~\cite{Horn}.

This paper is organized as follows.
In Section~2, we present our main results concerning the exact formulas for the AHTs and discuss their implications.
In Section~3, we provide fundamental definitions related to weighted graphs. 
In Section~4, we derive the exact formula for the AHTs of random walks on the weighted Cayley graph $\Cay(\ZZ_{2n},\{\pm1\})$ in a special case, by decomposing the coefficient matrix of the matrix equations for the AHTs into a product of simple matrices.
We also compute the inverse of each of these simple matrices to derive the exact formula.
In Section~5, we present the exact formula for the AHTs of random walks on the weighted Cayley graph $\Cay(\ZZ_N,\{+1,+2\})$ in a special case.
Finally, in Section~6, we establish the exact formula for the Kirchhoff index of random walks on these weighted Cayley graphs.

\section{Main Results}
In this section, we present our main results concerning the exact formulas for the AHTs on the weighted Cayley graphs.

\begin{thm} \label{main1}
For $1\leq \ell\leq 2n-1$, the exact formula for the AHTs of random walks on the weighted Cayley graph $\Cay(\ZZ_{2n},\{\pm1\})$ is
\begin{align*}
&h(\Cay(\ZZ_{2n},\{\pm1\});0,\ell) \\
&=\frac{1}{4p(1-p)}
\begin{cases}
\ell(2n-\ell) &\text{if $\ell$ is even,}\\
(\ell-1)(2n-\ell+1)+4(1-p)(n-\ell+p) &\text{if $\ell$ is odd,}
\end{cases}\\
&h(\Cay(\ZZ_{2n},\{\pm1\});1,\ell+1) \\
&=\frac{1}{4p(1-p)}
\begin{cases}
\ell(2n-\ell) &\text{if $\ell$ is even,}\\
(\ell-1)(2n-\ell+1)+4p(n-\ell+1-p) &\text{if $\ell$ is odd.}
\end{cases}
\end{align*}
\end{thm}

\begin{thm} \label{main2}
For $1\leq \ell\leq N-1$, an exact formula for the AHTs of random walks on the weighted Cayley graph $\Cay(\ZZ_N,\{+1,+2\})$ is
\[
h(\Cay(\ZZ_N,\{+1,+2\});0,\ell)
=\frac{N(p-1)((p-1)^\ell-1)-\ell((p-1)^N-1)}{(p-2)((p-1)^N-1)}.
\]
\end{thm}

These exact formulas under the weighted settings generalize the previous results for unweighted graphs provided by Doi et al.~\cite{HT} and Tanaka~\cite{YT2024}.
Crucially, our formulas explicitly capture the dependence on the weight parameters $p$ and $q$, revealing how the asymmetry of edge weights influences the expected random walk steps.
As an important application of these main theorems, we can also evaluate the topological invariants of the underlying graphs.
In Section~6, we utilize the above theorems to derive exact formulas for the Kirchhoff index of these weighted Cayley graphs, demonstrating the utility of our AHT formulas in electrical network theory.

\section{Preliminaries}

A directed graph (or digraph) is a couple $G=(V,E)$, where $V$ is a set of {\it vertices} and $E$ is a family of ordered pairs $(x,y)$ of distinct elements in $V$, which is called the set of {\it edges}. If $(x,y)\in E$, then we say that vertices $x$ and $y$ are {\it incident} with the edge $(x,y)$ {\it joining} $x$ and $y$.

The {\it Laplacian matrix} $\Lap_G=[\Lap_G(i,j)]$ indexed by the vertices for a weighted graph $G$ is defined as
\[
\Lap_G(i,j)=\begin{cases}
\sum_{x\in V}w((i,x)) &\text{if $i=j$,}\\
-w((i,j)) &\text{otherwise.}
\end{cases}
\]
Let $\Lap_G^{\prime}$ (resp.\ $\Lap_G^{\prime\prime}$) be the matrix obtained from $\Lap_G$ by deleting the first (resp.\ last) row and column. 


Let $G$ be a finite group and let $S\subseteq G$ be a subset.
The corresponding {\it Cayley graph} $\Cay({\it G,S})$ has its vertex set equal to $G$.
Two vertices $g,h\in G$ are joined by a directed edge from $g$ to $h$ if and only if there exists $s\in S$ such that $g=sh$.
For any vertex $x$ of a Cayley graph $\Cay({\it G,S})$, the out-degree $\deg(x)=|S|$ is constant.

\section{Weighted Cayley Graph $\Cay(\ZZ_{2n},\{\pm1\})$}

We define a weight $w(\{v_i,v_j\})$ on the weighted undirected Cayley graph $\Cay(\ZZ_{2n},\{\pm1\})$ as follows:
\[
w(\{v_i,v_j\})=
\begin{cases}
p &\text{if $j-i=1$, $i$ is even,}\\
q &\text{if $j-i=1$, $i$ is odd,}\\
0 &\text{otherwise,}
\end{cases}
\]
where $v_i=i+2n\ZZ\in\ZZ_{2n}$, and for any adjacent vertices, the integer representatives $i,j\in\ZZ$ are chosen such that $j-i=1$.

Let $\vec{h}$ be the column vector whose $i$-th entry is 
\[
\begin{cases}
h_{2n}(1,2n-i+1) &\text{if $1\leq i\leq 2n-1$,\;$i$ is odd,}\\
h_{2n}(0,2n-i) &\text{if $1\leq i\leq 2n-1$,\;$i$ is even,}\\
h_{2n}(1,4n-i) &\text{if $2n\leq i\leq 2(2n-1)$,\;$i$ is odd,}\\
h_{2n}(0,4n-i-1) &\text{if $2n\leq i\leq 2(2n-1)$,\;$i$ is even,}
\end{cases}
\]
and let $\vec{1}$ be a column vector of proper dimensions whose entries are all $1$.

On the weighted undirected Cayley graph $\Cay(\ZZ_{2n},\{\pm1\})$, we can assume that a random walker starts from vertex $0$ or $1$ without loss of generality.
By definition, in a random walk on $\Cay(\ZZ_{2n},\{\pm1\})$, the walker moves with probability $\frac{p}{p+q}$ or $\frac{q}{p+q}$ to an arbitrary adjacent vertex.
Moreover, by the symmetry of $\Cay(\ZZ_{2n},\{\pm1\})$, we have 
\begin{align*}
h_{2n}(0,\ell)&=h_{2n}(2k,2k+\ell)=h_{2n}(2k,2k-2n+\ell),\\
h_{2n}(1,\ell)&=h_{2n}(2k+1,2k+\ell)=h_{2n}(2k+1,2k-2n+\ell)
\end{align*}
for all $k,\ell \in\mathbb{Z}_{2n}$.
By the above properties, we have that $\forall \ell \in\mathbb{Z}_{2n}\setminus\{0\}$, 
\[
h_{2n}(0,\ell)=\frac{1}{p+q}\left(q(1+h_{2n}(-1,\ell))+p(1+h_{2n}(1,\ell))\right),
\]
and hence $\forall \ell \in\mathbb{Z}_{2n}\setminus\{0\}$,
\[
(p+q)h_{2n}(0,\ell)-qh_{2n}(1,\ell+2)-ph_{2n}(1,\ell)=p+q.
\] 

Similarly, we have that $\forall \ell \in\mathbb{Z}_{2n}\setminus\{1\}$, 
\[
h_{2n}(1,\ell)=\frac{1}{p+q}\left(p(1+h_{2n}(0,\ell))+q(1+h_{2n}(2,\ell))\right),
\]
and hence $\forall \ell \in\mathbb{Z}_{2n}\setminus\{1\}$,
\[
(p+q)h_{2n}(1,\ell)-ph_{2n}(0,\ell)-qh_{2n}(0,\ell-2)=p+q.
\] 
Therefore, we obtain
\[
\begin{bmatrix}
\Lap^{\prime}_{\Cay(\ZZ_{2n},\{\pm1\})}&O \\
O&\Lap^{\prime\prime}_{\Cay(\ZZ_{2n},\{\pm1\})}
\end{bmatrix}\vec{h}=(p+q)\vec{1}.
\]
Let $H_{2n}$ be the $2(2n-1)\times 2(2n-1)$ matrix whose $(i,j)$-th entry is
\[
\begin{cases}
\Lap^{\prime}_{\Cay(\ZZ_{2n},\{\pm1\})}(i,j) &\text{if $1\leq i,j\leq 2n-1$,}\\
\Lap^{\prime\prime}_{\Cay(\ZZ_{2n},\{\pm1\})}(i-2n+1,j-2n+1) &\text{if $2n\leq i,j\leq 2(2n-1)$,}\\
0 &\text{otherwise.}
\end{cases}
\]
Then we have $H_{2n} \vec{h}=(p+q)\vec{1}$.

Our goal is to obtain an exact formula for the AHTs of a random walk on $\Cay(\ZZ_{2n},\{\pm1\})$, so we need to solve the above matrix equation.

\subsection{Proof of Theorem~\ref{main1}}

To determine the AHTs, it suffices to solve the matrix equation $H_{2n}\vec{h}=(p+q)\vec{1}$.
First, we decompose the coefficient matrix $H_{2n}$.

Assume that $p+q=1$. Let $U_{2n}$ be the $2(2n-1) \times 2(2n-1)$ upper triangular matrix whose $(i,j)$-th entry is $1$ if $j \geq i$ and $0$ otherwise. Furthermore, we define the matrix $R_{2n}$ as follows:
\[
R_{2n}(i,j):=
\begin{cases}
2 &\text{if $i$ is odd,\;$1\leq i=j<2n$,}\\
3-2p &\text{if $i$ is even,\;$1\leq i=j<2n$,}\\
2p &\text{if $i$ is odd,\;$2n\leq i=j\leq 2(2n-1)$,}\\
2-p &\text{if $i\neq j$,\;$1\leq i,j<2n$,}\\
p &\text{if $i\neq j$,\;$2n<i\leq 2(2n-1)$,}\\
 &\text{ $i\neq j$,\;$2n<j\leq 2(2n-1)$,}\\
1 &\text{otherwise.}
\end{cases}
\]

Then the inverse matrix $R_{2n}^{-1}$ is given by 
\[
R_{2n}^{-1}(i,j)=\frac{1}{n}
\begin{cases}
\frac{n-(1-p)}{p} &\text{if $i=j$ and $i$ is odd,}\\
\frac{n-p}{1-p} &\text{if $i=j \neq 2n$ and $i$ is even,}\\
\frac{2(n-p)}{1-p} &\text{if $i=j=2n$,}\\
-\frac{1-p}{p} &\text{if $i$ and $j$ are odd, and either $i<2n<j$ or $j<2n<i$,}\\
-\frac{p}{1-p} &\text{if $i$ and $j$ are even, $i\neq j$, and either $i\leq 2n\leq j$ or $j\leq 2n\leq i$,}\\
-1 &\text{if $i\not\equiv j\pmod2$, and either $i\leq 2n\leq j$ or $j\leq 2n\leq i$,}\\
0 &\text{otherwise.}
\end{cases}
\]

Then we have the following.
\begin{thm}\label{decomposition}
$H_{2n}=U^{-1}_{2n}\,R_{2n}\,{}^tU^{-1}_{2n}$. 
\end{thm}

\begin{proof}
Let $S_{2n}$ denote the matrix $U_{2n}^{-1}R_{2n}{}^tU_{2n}^{-1}$.
By simple matrix computations, the entries of $S_{2n}$ can be calculated as follows: 
\small
\begin{align*}
&S_{2n}(i,j) \\
&=\begin{cases}
R_{2n}(i,j)-R_{2n}(i,j+1)-R_{2n}(i+1,j)+R_{2n}(i+1,j+1) \\
&\text{if $1\leq i,j<2(2n-1)$,}\\
R_{2n}(2(2n-1),j)-R_{2n}(2(2n-1),j+1) &\text{if $i=2(2n-1)$,}\\
 &\quad \text{$j\neq2(2n-1)$,}\\
R_{2n}(i,2(2n-1))-R_{2n}(i,2(2n-1)) &\text{if $i\neq2(2n-1)$,}\\
 &\quad \text{$j=2(2n-1)$,}\\
R_{2n}(2(2n-1),2(2n-1)) &\text{if $i=j=2(2n-1)$,} 
\end{cases}
\\
&=\begin{cases}
1 &\text{if $i=j$,}\\
-p &\text{if $i=j+1$,\;$i$ is odd,}\\
	&\text{$i=j-1$,\;$i$ is even,}\\
-1+p &\text{if $i=j+1,i\neq 2n$,\;$i$ is even,}\\
	&\text{$i=j-1,\;i\neq 2n-1$,\;$i$ is odd,}\\
0 &\text{otherwise.}
\end{cases}
\end{align*}
\normalsize

Therefore, we have $S_{2n}=H_{2n}$.
\end{proof}

Defining a new variable vector $\vec{h'}={}^tU_{2n}^{-1}\vec{h}$, the matrix equation $H_{2n} \vec{h}=\vec{1}$ can be expressed by a single matrix equation as 
\begin{equation*}
\vec{h'}=
\begin{bmatrix}
h'_1\\
h'_2\\
\vdots\\
h'_{2n-1}\\
h'_{2n}\\
h'_{2n+1}\\
\vdots\\
h'_{2(2n-1)-1}\\
h'_{2(2n-1)}\\
\end{bmatrix}
:=
\begin{bmatrix}
h_{2n}(1,2n)\\
h_{2n}(0,2n-2)-h_{2n}(1,2n)\\
\vdots\\
h_{2n}(1,2)-h_{2n}(0,2)\\
h_{2n}(0,2n-1)-h_{2n}(1,2)\\
h_{2n}(1,2n-1)-h_{2n}(0,2n-1)\\
\vdots\\
h_{2n}(1,3)-h_{2n}(0,3)\\
h_{2n}(0,1)-h_{2n}(1,3)\\
\end{bmatrix}
=R_{2n}^{-1}
\begin{bmatrix}
2(2n-1)\\
2(2n-1)-1\\
\vdots\\
2n\\
2n-1\\
2n-2\\
\vdots\\
2\\
1\\
\end{bmatrix}.
\end{equation*}

Solving this matrix equation, we obtain the following lemma. 
\begin{lem}\label{2.2}
For $1\leq \ell\leq 2(2n-1)$,
\[
h'_\ell=\begin{cases}
\frac{n-\ell+p}{p} &\text{if $1\leq \ell\leq 2n-1$,\;$\ell$ is odd,}\\
\frac{n-\ell+p}{1-p} &\text{if $1\leq \ell\leq 2n-1$,\;$\ell$ is even,}\\
0 &\text{if $\ell=2n$,}\\
\frac{3n-\ell-p}{p} &\text{if $2n+1\leq \ell\leq 2(2n-1)$,\;$\ell$ is odd,}\\
\frac{3n-\ell-p}{1-p} &\text{if $2n+1\leq \ell\leq 2(2n-1)$,\;$\ell$ is even.}
\end{cases}
\]
\end{lem}

\begin{proof}
For the case where $1\leq \ell\leq 2n-1$ ($\ell$ is odd), we have
\begin{align*}
h'_\ell&=
R_{2n}^{-1}{}^t
\begin{bmatrix}
4n-2, 4n-3, \cdots, 1
\end{bmatrix} \\
&=\sum_{k=1}^{2n-1}R_{2n}^{-1}(\ell,k)(4n-k-1)\\
&=\frac{n-(1-p)}{np}(4n-\ell-1)-\frac{1-p}{np}\sum_{\substack{1\leq k\leq n \\ k\neq\frac{\ell+1}{2}}}2(2n-k)-\frac{1}{n}\sum_{k=1}^n(2(2n-k)-1)\\
&=\frac{4n-\ell-1}{p}-\frac{1-p}{np}\sum_{k=1}^n2(2n-k)-\frac{1}{n}\sum_{k=1}^n(2(2n-k)-1)\\
&=\frac{4n-\ell-1}{p}-\frac{1-p}{np}(4n^2-n(n+1))-\frac{1}{n}(n(4n-1)-n(n+1))\\
&=\frac{n-\ell+p}{p}.
\end{align*}

For the case where $1\leq \ell\leq 2n-1$ ($\ell$ is even), we have
\begin{align*}
h'_\ell&=
R_{2n}^{-1}{}^t
\begin{bmatrix}
4n-2, 4n-3, \cdots, 1
\end{bmatrix} \\
&=\sum_{k=1}^{2n-1}R_{2n}^{-1}(\ell,k)(4n-k-1)\\
&=\frac{n-p}{n(1-p)}(4n-\ell-1)-\frac{p}{n(1-p)}\sum_{\substack{1\leq k\leq n \\ k\neq\frac{\ell}{2}}}(2(2n-k)-1)-\frac{1}{n}\sum_{k=1}^n2(2n-k)\\
&=\frac{4n-\ell-1}{1-p}-\frac{p}{n(1-p)}\sum_{k=1}^n(2(2n-k)-1)-\frac{1}{n}\sum_{k=1}^n2(2n-k)\\
&=\frac{4n-\ell-1}{1-p}-\frac{p}{n(1-p)}(n(4n-1)-n(n+1))-\frac{1}{n}(4n^2-n(n+1))\\
&=\frac{n-\ell+p}{1-p}.
\end{align*}

In the case of $\ell=2n$, we have
\begin{align*}
h'_\ell&=
R_{2n}^{-1}{}^t
\begin{bmatrix}
4n-2, 4n-3, \cdots, 1
\end{bmatrix} \\
&=\sum_{k=1}^{4n-3}R_{2n}^{-1}(2n,k)(4n-k-2)\\
&=\frac{2(n-p)}{n(1-p)}(2n-1)-\frac{p}{n(1-p)}\sum_{\substack{1\leq k\leq 2n-1 \\ k\neq n}}(2(2n-k)-1)-\frac{1}{n}\sum_{k=1}^{2n-1}2(2n-k)\\
&=\frac{(2n-1)(2n-p)}{n(1-p)}-\frac{p}{n(1-p)}((2n-1)(4n-1)-2n(2n-1))\\
&~-\frac{1}{n}(4n(2n-1)-2n(2n-1))\\
&=\frac{(2n-1)(2n-p)}{n(1-p)}-\frac{(2n-1)^2p}{n(1-p)}-2(2n-1)\\
&=0.
\end{align*}

For the case where $2n+1\leq \ell\leq 2(2n-1)$ ($\ell$ is odd), we have
\begin{align*}
h'_\ell&=
R_{2n}^{-1}{}^t
\begin{bmatrix}
4n-2, 4n-3, \cdots, 1
\end{bmatrix} \\
&=\sum_{k=2n-1}^{4n-3}R_{2n}^{-1}(\ell,k)(4n-k-1)\\
&=\frac{n-(1-p)}{np}(4n-\ell-1)-\frac{1-p}{np}\sum_{\substack{n+1\leq k\leq 2n-1 \\ k\neq\frac{\ell+1}{2}}}2(2n-k)-\frac{1}{n}\sum_{k=n}^{2n-1}(2(2n-k)-1)\\
&=\frac{4n-\ell-1}{p}-\frac{1-p}{np}\sum_{k=n+1}^{2n-1}2(2n-k)-\frac{1}{n}\sum_{k=n}^{2n-1}(2(2n-k)-1)\\
&=\frac{4n-\ell-1}{p}-\frac{1-p}{np}\sum_{k=1}^{n-1}2(n-k)-\frac{1}{n}\sum_{k=0}^{n-1}(2(n-k)-1)\\
&=\frac{4n-\ell-1}{p}-\frac{1-p}{np}(2n(n-1)-n(n-1))-\frac{1}{n}(n(2n-1)-n(n-1))\\
&=\frac{3n-\ell-p}{p}.
\end{align*}

For the case where $2n+1\leq \ell\leq 2(2n-1)$ ($\ell$ is even), we have
\begin{align*}
h'_\ell&=
R_{2n}^{-1}{}^t
\begin{bmatrix}
4n-2, 4n-3, \cdots, 1
\end{bmatrix} \\
&=\sum_{k=2n-1}^{4n-3}R_{2n}^{-1}(\ell,k)(4n-k-1)\\
&=\frac{n-p}{n(1-p)}(4n-\ell-1)-\frac{p}{n(1-p)}\sum_{\substack{n\leq k\leq 2n-1 \\ k\neq\frac{\ell}{2}}}(2(2n-k)-1)-\frac{1}{n}\sum_{k=n+1}^{2n-1}2(2n-k)\\
&=\frac{4n-\ell-1}{1-p}-\frac{p}{n(1-p)}\sum_{k=n}^{2n-1}(2(2n-k)-1)-\frac{1}{n}\sum_{k=n+1}^{2n-1}2(2n-k)\\
&=\frac{4n-\ell-1}{1-p}-\frac{p}{n(1-p)}\sum_{k=0}^{n-1}(2(n-k)-1)-\frac{1}{n}\sum_{k=1}^{n-1}2(n-k)\\
&=\frac{4n-\ell-1}{1-p}-\frac{p}{n(1-p)}(n(2n-1)-n(n-1))-\frac{1}{n}(2n(n-1)-n(n-1))\\
&=\frac{3n-\ell-p}{1-p}.
\end{align*}
\end{proof}

Combining Lemma~\ref{2.2} and the equations 
$h_{0,\ell}=\begin{cases}
\sum_{i=1}^{2n-\ell}h'_i &\text{if $\ell$ is even,}\\
\sum_{i=1}^{4n-\ell-1}h'_i &\text{if $\ell$ is odd,}
\end{cases}$ and $h_{1,\ell+1}=\begin{cases}
\sum_{i=1}^{4n-\ell-1}h'_i &\text{if $\ell$ is even,}\\
\sum_{i=1}^{2n-\ell}h'_i &\text{if $\ell$ is odd,}
\end{cases}$ 
for $1\leq \ell\leq 2n-1$, we obtain Theorem~\ref{main1}.

\begin{proof}

For the case where $\ell$ is even, we have

\begin{align*}
h_{2n}(0,\ell)&=\sum_{i=1}^{\frac{2n-\ell}{2}}\frac{n-(2i-1)+p}{p}+\sum_{i=1}^{\frac{2n-\ell}{2}}\frac{n-2i+p}{1-p}\\
&=\sum_{i=1}^{\frac{2n-\ell}{2}}\frac{1}{p}+\sum_{i=1}^{\frac{2n-\ell}{2}}\frac{n-2i+p}{p(1-p)}\\
&=\frac{2n-\ell}{2p}+\frac{(2n-\ell)(n+p)}{2p(1-p)}-\frac{(2n-\ell)(2n-\ell+2)}{4p(1-p)}\\
&=\frac{\ell(2n-\ell)}{4p(1-p)}.
\end{align*}

For the case where $\ell$ is odd ($\ell\neq1,2n-1$), we have
\begin{align*}
h_{2n}(0,\ell)&=\sum_{i=1}^n\frac{n-(2i-1)+p}{p}+\sum_{i=1}^{n-1}\frac{n-2i+p}{1-p}\\
&~+\sum_{i=n+1}^{\frac{4n-\ell-1}{2}}\frac{3n-(2i-1)-p}{p}+\sum_{i=n+1}^{\frac{4n-\ell-1}{2}}\frac{3n-2i-p}{1-p}\\
&=\sum_{i=1}^n\frac{n-(2i-1)+p}{p}+\sum_{i=1}^{n-1}\frac{n-2i+p}{1-p}\\
&~+\sum_{i=1}^{\frac{2n-\ell-1}{2}}\frac{n-(2i-1)-p}{p}+\sum_{i=1}^{\frac{2n-\ell-1}{2}}\frac{n-2i-p}{1-p}\\
&=\sum_{i=1}^n\frac{1}{p}-\frac{n-p}{p}+\sum_{i=1}^{n-1}\frac{n-2i+p}{p(1-p)}+\sum_{i=1}^{\frac{2n-\ell-1}{2}}\frac{1}{p}+\sum_{i=1}^{\frac{2n-\ell-1}{2}}\frac{n-2i-p}{p(1-p)}\\
&=\frac{n}{p}-\frac{n-p}{p}+\frac{(n-1)(n+p)}{p(1-p)}-\frac{n(n-1)}{p(1-p)}\\
&~+\frac{2n-\ell-1}{2p}+\frac{(2n-\ell-1)(n-p)}{2p(1-p)}-\frac{(2n-\ell-1)(2n-\ell+1)}{4p(1-p)}\\
&=\frac{2n+2p-\ell-1}{2p}+\frac{p(n-1)}{p(1-p)}+\frac{(2n-\ell-1)(-2p+\ell-1)}{4p(1-p)}\\
&=\frac{(\ell-1)(2n-\ell+1)+4(1-p)(n-\ell+p)}{4p(1-p)}.
\end{align*}

For the case where $\ell=2n-1$, we have
\begin{align*}
h_{2n}(0,2n-1)
&=\sum_{i=1}^n\frac{n-(2i-1)+p}{p}+\sum_{i=1}^{n-1}\frac{n-2i+p}{1-p}\\
&=\sum_{i=1}^n\frac{1}{p}-\frac{n-p}{p}+\sum_{i=1}^{n-1}\frac{n-2i+p}{p(1-p)}\\
&=\frac{n}{p}-\frac{n-p}{p}+\frac{(n+p)(n-1)}{p(1-p)}-\frac{n(n-1)}{p(1-p)}\\
&=\frac{(n-1)-(1-p)(n-p-1)}{p(1-p)} \\
&=\frac{(2n-1-1)(2n-(2n-1)+1)+4(1-p)(n-(2n-1)+p)}{4p(1-p)}.
\end{align*}

For the case where $\ell=1$, we have
\begin{align*}
h_{2n}(0,1)&=\sum_{i=1}^n\frac{n-(2i-1)+p}{p}+\sum_{i=1}^{n-1}\frac{n-2i+p}{1-p}\\
&~+\sum_{i=n+1}^{2n-1}\frac{3n-(2i-1)-p}{p}+\sum_{i=n+1}^{2n-1}\frac{3n-2i-p}{1-p}\\
&=\sum_{i=1}^n\frac{n-(2i-1)+p}{p}+\sum_{i=1}^{n-1}\frac{n-2i+p}{1-p}\\
&~+\sum_{i=1}^{n-1}\frac{n-(2i-1)-p}{p}+\sum_{i=1}^{n-1}\frac{n-2i-p}{1-p}\\
&=\sum_{i=1}^n\frac{1}{p}-\frac{n-p}{p}+\sum_{i=1}^{n-1}\frac{n-2i+p}{p(1-p)}+\sum_{i=1}^{n-1}\frac{1}{p}+\sum_{i=1}^{n-1}\frac{n-2i-p}{p(1-p)}\\
&=\frac{n}{p}-\frac{n-p}{p}+\frac{2n(n-1)}{p(1-p)}-\frac{2n(n-1)}{p(1-p)}+\frac{n-1}{p}\\
&=\frac{n-1+p}{p}\\
&=\frac{(1-1)(2n-1+1)+4(1-p)(n-1+p)}{4p(1-p)}.
\end{align*}

For the case where $\ell$ is even, we have
\begin{align*}
h_{2n}(1,\ell+1)&=\sum_{i=1}^n\frac{n-(2i-1)+p}{p}+\sum_{i=1}^{n-1}\frac{n-2i+p}{1-p}\\
&~+\sum_{i=n+1}^{\frac{4n-\ell}{2}}\frac{3n-(2i-1)-p}{p}+\sum_{i=n+1}^{\frac{4n-\ell-2}{2}}\frac{3n-2i-p}{1-p}\\
&=\sum_{i=1}^n\frac{n-(2i-1)+p}{p}+\sum_{i=1}^{n-1}\frac{n-2i+p}{1-p}\\
&~+\sum_{i=1}^{\frac{2n-\ell}{2}}\frac{n-(2i-1)-p}{p}+\sum_{i=1}^{\frac{2n-\ell-2}{2}}\frac{n-2i-p}{1-p} \\
&=\sum_{i=1}^n\frac{1}{p}-\frac{n-p}{p}+\sum_{i=1}^{n-1}\frac{n-2i+p}{p(1-p)}\\
&~+\sum_{i=1}^{\frac{2n-\ell}{2}}\frac{1}{p}+\frac{n-(2n-\ell)-p}{p}+\sum_{i=1}^{\frac{2n-\ell-2}{2}}\frac{n-2i-p}{p(1-p)}\\
&=\frac{n}{p}-\frac{n-p}{p}+\frac{(n-1)(n+p)}{p(1-p)}-\frac{n(n-1)}{p(1-p)}\\
&~+\frac{2n-\ell}{2p}-\frac{n-\ell+p}{p}+\frac{(2n-\ell-2)(n-p)}{2p(1-p)}-\frac{(2n-\ell-2)(2n-\ell)}{4p(1-p)}\\
&=\frac{p(n-p)}{p(1-p)}+\frac{\ell-2p}{2p}+\frac{(2n-\ell-2)(\ell-2p)}{4p(1-p)}\\
&=\frac{\ell(2n-\ell)}{4p(1-p)}.
\end{align*}

For the case where $\ell$ is odd, we have
\begin{align*}
h_{2n}(1,\ell+1)
&=\sum_{i=1}^{\frac{2n-\ell+1}{2}}\frac{n-(2i-1)+p}{p}+\sum_{i=1}^{\frac{2n-\ell-1}{2}}\frac{n-2i+p}{1-p}\\
&=\sum_{i=1}^{\frac{2n-\ell+1}{2}}\frac{1}{p}-\frac{n-\ell-p+1}{p}+\sum_{i=1}^{\frac{2n-\ell-1}{2}}\frac{n-2i+p}{p(1-p)}\\
&=\frac{2n-\ell+1}{2p}-\frac{n-\ell-p+1}{p}+\frac{(2n-\ell-1)(n+p)}{2p(1-p)}\\
&~-\frac{(2n-\ell-1)(2n-\ell+1)}{4p(1-p)}\\
&=\frac{\ell-1+2p}{2p}+\frac{(2n-\ell-1)(\ell-1)+2p(2n-\ell-1)}{4p(1-p)}\\
&=\frac{(\ell-1)(2n-\ell+1)+4p(n-\ell+1-p)}{4p(1-p)}.
\end{align*}
\end{proof}

From Theorem~\ref{decomposition}, we have $H_{2n}^{-1}={}^tU_{2n}\,R_{2n}^{-1}\,U_{2n}$. 
Hence, the entries of $H^{-1}_{2n}$ can also be calculated as follows. 
For $1\leq i,j\leq 2(2n-1)$, we have 
\begin{align*}
&H_{2n}^{-1}(i,j)=\frac{1}{4np(1-p)}\times\\
&~\begin{cases}
(2n-i-1+2p)(j+1-2p) &\text{if $1\leq j\leq i\leq 2n-1$,\;$i$ is odd,\;$j$ is odd,}\\
(2n-i)(j+1-2p) &\text{if $1\leq j\leq i\leq 2n-1$,\;$i$ is even,\;$j$ is odd,}\\
(2n-i-1+2p)j &\text{if $1\leq j\leq i\leq 2n-1$,\;$i$ is odd,\;$j$ is even,}\\
(2n-i)j &\text{if $1\leq j\leq i\leq 2n-1$,\;$i$ is even,\;$j$ is even,}\\
(2n-j-1+2p)(i+1-2p) &\text{if $1\leq i< j\leq 2n-1$,\;$i$ is odd,\;$j$ is odd,}\\
(2n-j-1+2p)i &\text{if $1\leq i< j\leq 2n-1$,\;$i$ is even,\;$j$ is odd,}\\
(2n-j)(i+1-2p) &\text{if $1\leq i< j\leq 2n-1$,\;$i$ is odd,\;$j$ is even,}\\
(2n-j)i &\text{if $1\leq i< j\leq 2n-1$,\;$i$ is even,\;$j$ is even,}\\
(2n-4n+i+2p)(4n-j-2p) &\text{if $2n\leq j\leq i\leq 2(2n-1)$,\;$i$ is odd,\;$j$ is odd,}\\
(2n-4n+i+1)(4n-j-2p) &\text{if $2n\leq j\leq i\leq 2(2n-1)$,\;$i$ is even,\;$j$ is odd,}\\
(2n-4n+i+2p)(4n-j-1)&\text{if $2n\leq j\leq i\leq 2(2n-1)$,\;$i$ is odd,\;$j$ is even,}\\
(2n-4n+i+1)(4n-j-1) &\text{if $2n\leq j\leq i\leq 2(2n-1)$,\;$i$ is even,\;$j$ is even,}\\
(2n-4n+j+2p)(4n-i-2p) &\text{if $2n\leq i<j\leq 2(2n-1)$,\;$i$ is odd,\;$j$ is odd,}\\
(2n-4n+j+2p)(4n-i-1) &\text{if $2n\leq i<j\leq 2(2n-1)$,\;$i$ is even,\;$j$ is odd,}\\
(2n-4n+j+1)(4n-i-2p) &\text{if $2n\leq i<j\leq 2(2n-1)$,\;$i$ is odd,\;$j$ is even,}\\
(2n-4n+j+1)(4n-i-1) &\text{if $2n\leq i<j\leq 2(2n-1)$,\;$i$ is even,\;$j$ is even,}\\
0 &\text{otherwise.}
\end{cases}
\end{align*}

\section{Weighted Cayley Graph $\Cay(\ZZ_N,\{+1,+2\})$}

We define a weight $w((v_i,v_j))$ on the weighted directed Cayley graph $\Cay(\ZZ_N,\{+1,+2\})$ as follows: 

\[
w((v_i,v_j))=
\begin{cases}
p &\text{if $j-i=1$,}\\
q &\text{if $j-i=2$,}\\
0 &\text{otherwise,}
\end{cases}
\]
where $i,j\in\ZZ$ and $v_i=i+N\ZZ\in\ZZ_N$.

Let $h_N(k,\ell)$ denote the AHT from a vertex $v_k$ to a vertex $v_\ell$ on the weighted directed Cayley graph $\Cay(\ZZ_N,\{+1,+2\})$.
Let $\vec{h}$ be the $(N-1)$-dimensional column vector whose $i$-th entry is $h_N(0,i)$ for $1 \leq i \leq N-1$, and let $\vec{1}$ be a column vector of proper dimensions whose entries are all $1$.

On the weighted directed Cayley graph $\Cay(\ZZ_N,\{+1,+2\})$, we can assume that a random walker starts from vertex $0$ without loss of generality.
By definition, the walker moves with probability $\frac{p}{p+q}$ or $\frac{q}{p+q}$ to an arbitrary vertex adjacent to the current vertex.
Moreover, by the symmetry of $\Cay(\ZZ_N,\{+1,+2\})$, we have
\[
h_N(0,\ell)=h_N(k,k+\ell)=h_N(k+N-\ell,k) \quad \text{for all $k,\ell \in \mathbb{Z}_{N}$.}
\]
Note that $h_N(0,0)=0$.
By the above properties, for any $\ell \in\mathbb{Z}_N\setminus\{0\}$, we have
\[
h_N(0,\ell)=\frac{p}{p+q}(1+h_N(1,\ell))+\frac{q}{p+q}(1+h_N(2,\ell)),
\]
and hence $\forall \ell \in\mathbb{Z}_N\setminus\{0\}$,
\[
(p+q)h_N(0,\ell)-ph_N(0,\ell-1)-qh_N(0,\ell-2)=p+q.
\] 
We thus have ${}^t\Lap_{\Cay(\ZZ_N,\{+1,+2\})}^{\prime\prime} \vec{h}=(p+q)\vec{1}$.
Let $H_N$ be the $(N-1)\times(N-1)$ matrix whose $(i,j)$-th entry is ${}^t\Lap_{\Cay(\ZZ_N,\{+1,+2\})}^{\prime\prime}(i,j)$.

Then we obtain the matrix equation $H_N\vec{h}=(p+q)\vec{1}$.

Our goal is to obtain an exact formula for the AHTs of random walks on $\Cay(\ZZ_N,\{+1,+2\})$, so we need to solve the above matrix equation.

\subsection{Proof of Theorem~\ref{main2}}

To determine the AHTs, it suffices to solve the matrix equation $H_N\vec{h}=(p+q)\vec{1}$.
First, we utilize $LU$ decomposition to analyze the coefficient matrix $H_N$.

\begin{df}[$LU$ decomposition]
Let $M_N$ denote the set of all $N\times N$ matrices over $\CC$.
For $A\in M_N$, a factorization $A=LU$, where $L\in M_N$ is a lower triangular matrix and $U\in M_N$ is an upper triangular matrix, is called an $LU$ decomposition of $A$.
\end{df}

Let $p+q=1$, and let $P_N$ be the matrix
\[
P_N:=\begin{bmatrix}
0&1&0&0&\cdots&0 \\
1&0&0&0&\cdots&0 \\
0&0&1&0&\cdots&0 \\
\vdots&\ddots&\ddots&\ddots&\ddots&\vdots \\
0&\cdots&0&0&1&0\\
0&\cdots&0&0&0&1
\end{bmatrix}.
\]

Let $L_N$ be the matrix
\[
L_N:=\begin{bmatrix}
1&0&0&0&\cdots&0&0\\
-\frac{1}{p}&1&0&0&0&\cdots&0\\
-\frac{-1+p}{p}&-1+p-p^2&1&0&0&\cdots&0\\
0&p(-1+p)&-p&1&0&\ddots&\vdots\\
0&0&-1+p&-p&1&\ddots&\vdots\\
\vdots&\vdots&\ddots&\ddots&\ddots&\ddots&0\\
0&0&\cdots&0&-1+p&-p&1
\end{bmatrix}.
\]

Then, its inverse $L_N^{-1}$ is 
\[
L_N^{-1}=\begin{bmatrix}
1&0&0&0&\cdots&0&0\\
\frac{1}{p}&1&0&0&0&\cdots&0\\
\frac{(p-1)^2-1}{p-2}&\frac{(p-1)^3-1}{p-2}&1&0&0&\cdots&0\\
\frac{(p-1)^3-1}{p-2}&\frac{(p-1)^4-1}{p-2}&\frac{(p-1)^2-1}{p-2}&1&0&\ddots&\vdots\\
\frac{(p-1)^4-1}{p-2}&\frac{(p-1)^5-1}{p-2}&\frac{(p-1)^3-1}{p-2}&\frac{(p-1)^2-1}{p-2}&1&\ddots&\vdots\\
\vdots&\vdots&\ddots&\ddots&\ddots&\ddots&0\\
\frac{(p-1)^{N-2}-1}{p-2}&\frac{(p-1)^{N-1}-1}{p-2}&\frac{(p-1)^{N-3}-1}{p-2}&\cdots&\frac{(p-1)^3-1}{p-2}&\frac{(p-1)^2-1}{p-2}&1
\end{bmatrix}.
\]

Let $U_N$ be the matrix
\[
U_N:=\begin{bmatrix}
-p&1&0&\cdots&0&0\\
0&\frac{1}{p}&0&\cdots&0&-1+p\\
0&0&1&\ddots&\vdots&\frac{(p-1)^{4}-p+1}{p-2}\\
0&0&0&\ddots&0&\vdots\\
\vdots&\vdots&\ddots&\ddots&1&\frac{(p-1)^{N}-p+1}{p-2}\\
0&0&\cdots&0&0&\frac{(p-1)^{N}-1}{p-2}
\end{bmatrix}.
\]

Then, its inverse $U_N^{-1}$ is 
\[
U_N^{-1}=\begin{bmatrix}
-\frac{1}{p}&1&0&\cdots&0&\frac{(p-1)^2-p+1}{1-(p-1)^N}\\
0&p&0&\cdots&0&\frac{(p-1)^3-p+1}{1-(p-1)^N}\\
0&0&1&\ddots&\vdots&\frac{(p-1)^4-p+1}{1-(p-1)^N}\\
0&0&0&\ddots&0&\vdots\\
\vdots&\vdots&\ddots&\ddots&1&\frac{(p-1)^{N}-p+1}{1-(p-1)^N}\\
0&0&\cdots&0&0&\frac{p-2}{(p-1)^N-1}
\end{bmatrix}.
\]

Then we have the following.
\begin{thm}\label{decomposition2}
$H_N=P_NL_NU_N$.
\end{thm}

\begin{proof}

Let $S_N$ denote the matrix $P_NL_NU_N$.
By simple matrix computation, the entries of $S_N$ can be calculated as 
\[
S_N(i,j)=\begin{cases}
1 &\text{if $i=j\neq N-1$,}\\
-p &\text{if $i=j+1$,\;$4\leq i\leq N-1$,}\\
-1+p &\text{if $i=j+2$,\;$3\leq i\leq N-1$,}\\
 &\text{   $(i,j)=(1,N-1)$,}\\
-\frac{1}{p}+\frac{1}{p} &\text{if $(i,j)=(1,2)$,}\\
-\frac{-1+p}{p}+\frac{-1+p-p^2}{p} &\text{if $(i,j)=(3,2)$,}\\
(-1+p)(-1+p-p^2)+\frac{(p-1)^4-p+1}{p-2} &\text{if $(i,j)=(3,N-1)$,}\\
p(-1+p)^2-\frac{p((p-1)^4-p+1)}{p-2}+\frac{(p-1)^5-p+1}{p-2} &\text{if $(i,j)=(4,N-1)$,}\\
(-1+p)\frac{(p-1)^{i-1}-p+1}{p-2}-p\frac{(p-1)^i-p+1}{p-2}+\frac{(p-1)^{i+1}-p+1}{p-2}  &\text{if $j=N-1$,\;$5\leq i\leq N-2$,}\\
(-1+p)\frac{(p-1)^{N-2}-p+1}{p-2}-p\frac{(p-1)^{N-1}-p+1}{p-2}+\frac{(p-1)^N-1}{p-2} &\text{if $i=j=N-1$,}\\
0 &\text{otherwise.}\\
\end{cases}
\]

For the case where $(i,j)=(1,2)$, we have
\[
S_N(i,j)=-\frac{1}{p}+\frac{1}{p}=0.
\]

For the case where $(i,j)=(3,2)$, we have
\[
S_N(i,j)=-\frac{-1+p}{p}+\frac{-1+p-p^2}{p}=-p.
\]

For the case where $(i,j)=(3,N-1)$, we have
\begin{align*}
S_N(i,j)&=(-1+p)(-1+p-p^2)+\frac{(p-1)^4-p+1}{p-2}\\
&=\frac{(p-1)^2(p-2)-p^2(p-1)(p-2)}{p-2}+\frac{(p-1)^4-p+1}{p-2}\\
&=\frac{(p-1)(p^2-3p+2-(p^3-2p^2)+(p^3-3p^2+3p-1)-1)}{p-2}\\
&=0.
\end{align*}

For the case where $(i,j)=(4,N-1)$, we have
\begin{align*}
S_N(i,j)&=p(-1+p)^2-\frac{p((p-1)^4-p+1)}{p-2}+\frac{(p-1)^5-p+1}{p-2}\\
&=\frac{p(p-1)^2(p-2)}{p-2}+\frac{-p(p-1)^4+p(p-1)+(p-1)^5-(p-1)}{p-2}\\
&=\frac{p(p-1)^2(p-2)-(p-1)^4+(p-1)^2}{p-2}\\
&=0.
\end{align*}

For the case where $j=N-1$,\;$5\leq i\leq N-2$, we have
\begin{align*}
S_N(i,j)&=(-1+p)\frac{(p-1)^{i-1}-p+1}{p-2}-p\frac{(p-1)^i-p+1}{p-2}+\frac{(p-1)^{i+1}-p+1}{p-2}\\
&=\frac{(p-1)^i-(p-1)^2-p(p-1)^i+p(p-1)+(p-1)^{i+1}-(p-1)}{p-2}\\
&=\frac{-(p-1)^{i+1}-(p-1)^2+(p-1)^2+(p-1)^{i+1}}{p-2}\\
&=0.
\end{align*}

For the case where $i=j=N-1$, we have
\begin{align*}
S_N(i,j)&=(-1+p)\frac{(p-1)^{N-2}-p+1}{p-2}-p\frac{(p-1)^{N-1}-p+1}{p-2}+\frac{(p-1)^N-1}{p-2}\\
&=\frac{(p-1)^{N-1}-(p-1)^2-p(p-1)^{N-1}+p(p-1)+(p-1)^N-1}{p-2}\\
&=\frac{-(p-1)^N-(p^2-2p+1)+p^2-p+(p-1)^N-1}{p-2}\\
&=1.
\end{align*}

Therefore, we have $S_N=H_N$.
\end{proof}

From Theorem~\ref{decomposition2}, the matrix equation $H_N \vec{h}=\vec{1}$ can be expressed by the single matrix equation 
\begin{equation*}
\vec{h}=
\begin{bmatrix}
h_N(0,1)\\
h_N(0,2)\\
\vdots\\
h_N(0,\ell)\\
\vdots\\
h_N(0,N-1)
\end{bmatrix}
=U_N^{-1}L_N^{-1}
\begin{bmatrix}
1\\
1\\
\vdots\\
1\\
\vdots\\
1
\end{bmatrix}.
\end{equation*}

Solving this matrix equation, we obtain an exact formula for $h_N(0,\ell)$. 

\begin{thm}\label{thm:+134}
For $1\leq \ell\leq N-1$, an exact formula for the AHTs of random walks on the weighted Cayley graph $\Cay(\ZZ_N,\{+1,+2\})$ is
\[
h_N(0,\ell)=\frac{N(p-1)((p-1)^\ell-1)-\ell((p-1)^N-1)}{(p-2)((p-1)^N-1)}.
\]
\end{thm}

\begin{proof}
For the case where $\ell=1$, we have
\begin{align*}
h_N(0,1)&=1+\frac{(p-1)^2-p+1}{1-(p-1)^N}\sum_{i=1}^{N-1}\frac{(p-1)^i-1}{p-2}\\
&=1+\frac{(p-1)^2-p+1}{(p-2)(1-(p-1)^N)}(\frac{(p-1)(1-(p-1)^{N-1})}{1-(p-1)}-N+1)\\
&=1-\frac{p-1}{(p-1)^N-1}(\frac{(p-1)^N-1}{p-2}-N)\\
&=\frac{(p-2)((p-1)^N-1)-(p-1)((p-1)^N-1)+N(p-1)(p-2)}{(p-2)((p-1)^N-1)}\\
&=\frac{N(p-1)((p-1)^1-1)-((p-1)^N-1)}{(p-2)((p-1)^N-1)}.
\end{align*}

For the case where $\ell=2$, we have
\begin{align*}
h_N(0,2)&=1+p+\frac{(p-1)^3-p+1}{1-(p-1)^N}\sum_{i=1}^{N-1}\frac{(p-1)^i-1}{p-2}\\
&=1+p+\frac{(p-1)^3-p+1}{(p-2)(1-(p-1)^N)}(\frac{(p-1)(1-(p-1)^{N-1})}{1-(p-1)}-N+1)\\
&=1+p-\frac{(p-1)((p-1)^2-1)}{(p-2)((p-1)^N-1)}(\frac{(p-1)^N-1}{p-2}-N)\\
&=\frac{(p+1)(p-2)^2}{(p-2)^2}-\frac{(p-1)((p-1)^2-1)}{(p-2)^2}+\frac{N(p-1)((p-1)^2-1)}{(p-2)((p-1)^N-1)}\\
&=\frac{(p+1)(p-2)^2}{(p-2)^2}-\frac{p(p-1)(p-2)}{(p-2)^2}+\frac{N(p-1)((p-1)^2-1)}{(p-2)((p-1)^N-1)}\\
&=\frac{(p+1)(p-2)-p(p-1)}{(p-2)}+\frac{N(p-1)((p-1)^2-1)}{(p-2)((p-1)^N-1)}\\
&=\frac{N(p-1)((p-1)^2-1)-2((p-1)^N-1)}{(p-2)((p-1)^N-1)}.
\end{align*}

For the case where $3\leq \ell\leq N-2$, we have
\begin{align*}
h_N(0,\ell)&=\sum_{i=1}^\ell\frac{(p-1)^i-1}{p-2}+\frac{(p-1)^{\ell+1}-p+1}{1-(p-1)^N}\sum_{i=1}^{N-1}\frac{(p-1)^i-1}{p-2}\\
&=\frac{1}{p-2}(\sum_{i=1}^\ell(p-1)^i-l)+\frac{(p-1)^{\ell+1}-p+1}{(p-2)(1-(p-1)^N)}(\sum_{i=1}^{N-1}(p-1)^i-N+1)\\
&=\frac{1}{p-2}(\frac{(p-1)(1-(p-1)^\ell)}{1-(p-1)}-\ell\\
&~+\frac{(p-1)^{\ell+1}-p+1}{1-(p-1)^N}(\frac{(p-1)(1-(p-1)^{N-1})}{1-(p-1)}-N+1))\\
&=\frac{1}{p-2}(\frac{(p-1)^{\ell+1}-p+1}{p-2}-\ell\\
&~-\frac{(p-1)^{\ell+1}-p+1}{(p-1)^N-1}(\frac{(p-1)^N-p+1}{p-2}-N+1))\\
&=\frac{1}{p-2}(\frac{(p-1)^{\ell+1}-p+1}{p-2}-\ell-\frac{(p-1)^{\ell+1}-p+1}{(p-1)^N-1}(\frac{(p-1)^N-1}{p-2}-N))\\
&=\frac{N(p-1)((p-1)^\ell-1)-\ell((p-1)^N-1)}{(p-2)((p-1)^N-1)}.
\end{align*}

For the case where $\ell=N-1$, we have
\begin{align*}
h_N(0,N-1)&=\frac{p-2}{(p-1)^{N}-1}\sum_{i=1}^{N-1}\frac{(p-1)^i-1}{p-2}\\
&=\frac{1}{(p-1)^{N}-1}(\frac{(p-1)(1-(p-1)^{N-1})}{1-(p-1)}-N+1)\\
&=\frac{-N(p-2)+(p-1)^N-1}{(p-2)((p-1)^N-1)}\\
&=\frac{N((p-1)^N-p+1)-N((p-1)^N-1)+(p-1)^N-1}{(p-2)((p-1)^N-1)}\\
&=\frac{N(p-1)((p-1)^{N-1}-1)-(N-1)((p-1)^N-1)}{(p-2)((p-1)^N-1)}.
\end{align*}
\end{proof}

From Theorem~\ref{decomposition2}, we have $H_N^{-1}=U_N^{-1}L_N^{-1}P_N^{-1}$, and so the entries of $H^{-1}_N$ can also be calculated as follows. 
For $1\leq i,j\leq N-1$, we have 

\noindent
\footnotesize
\begin{align*}
&H_N^{-1}(i,j)
=\frac{1}{(p-2)((p-1)^N-1)}\\
&~\begin{cases}
(p-2)^2 &\text{if $i=j=1$,}\\
(p-2)((p-1)^N-p(p-1)^{N-1}+p^2-p-1) &\text{if $i=j=2$,}\\
-(p-2)((p-1)^{N-1}-p+1) &\text{if $(i,j)=(1,2)$,}\\
p(p-2)^2 &\text{if $(i,j)=(2,1)$,}\\
-(p-2)((p-1)^{N-j+1}-p+1) &\text{if $i=1,\;3\leq j\leq N-1$,}\\
-p(p-2)((p-1)^{N-j+1}-p+1) &\text{if $i=2,\;3\leq j\leq N-1$,}\\
(p-2)((p-1)^i-1) &\text{if $3\leq i\leq N-2,\;j=1$,}\\
(p-2)(p(p-1)^{i-1}-(p-1)^{N-1}-1) &\text{if $3\leq i\leq N-2,\;j=2$,}\\
-((p-1)^i-1)((p-1)^{N-j+1}-p+1) &\text{if $3\leq i<j\leq N-1$,}\\
(p-2)((p-1)^i-1)+((p-1)^{-j+1}-1)((p-1)^{N}-(p-1)^{i}) &\text{if $3\leq j\leq i\leq N-2$,}\\
(p-2)((p-1)^{N-j}-1) &\text{if $i=N-1$.}
\end{cases}
\end{align*}
\normalsize

\section{Application: Effective Resistance and Kirchhoff Index}

The {\it effective resistance} between a pair of vertices $u$ and $v$ of a weighted graph $G=(V,E)$ is the electrical resistance seen between the vertices $u$ and $v$ of a resistor network with branch conductances given by the edge weights, denoted by $r(G;u,v)$.

In 1996, Chandra et al.~\cite{CRRST} proved the following formula relating the effective resistances and the AHTs.

\begin{thm}[Chandra et al., 1996~\cite{CRRST}]
\[
r(G;u,v)=\frac{h(G;u,v)+h(G;v,u)}{\sum_{(x,y)\in V\times V}w((x,y))}.
\]
\end{thm}

Let $V=\{v_1,v_2,\ldots,v_n\}$.
The {\it Kirchhoff index} $\Kf(G)$ of the 
graph $G$
 is defined as $\Kf(G):=\frac{1}{2}\sum_{u,v\in V}r(G;u,v)$.

We can obtain the following theorems.

\begin{thm}
We define a weight $w(\{v_i,v_j\})$ on the weighted undirected Cayley graph $\Cay(\ZZ_{2n},\{\pm1\})$ as follows:
\[
w(\{v_i,v_j\})=
\begin{cases}
p &\text{if $j-i=1$, $i$ is even,}\\
q &\text{if $j-i=1$, $i$ is odd,}\\
0 &\text{otherwise,}
\end{cases}
\]
where $i,j\in \ZZ$ and $v_i=i+2n\ZZ\in\ZZ_{2n}$.

Then, 
\[
\Kf(\Cay(\ZZ_{2n},\{\pm1\}))=\frac{n(n-1)(n+1)+3np(1-p)}{3p(1-p)}.
\]
\end{thm}

\begin{proof}

\begin{align*}
&\Kf(\Cay(\ZZ_{2n},\{\pm1\})) \\
&=\sum_{0\leq i<j\leq 2n-1}r(\Cay(\ZZ_{2n},\{\pm1\});v_i,v_j) \\
&=\sum_{0\leq i<j\leq 2n-1}\frac{h(\Cay(\ZZ_{2n},\{\pm1\});i,j)+h(\Cay(\ZZ_{2n},\{\pm1\});j,i)}{\sum_{0\leq k,\ell\leq 2n-1}w((v_k,v_\ell))} \\
&=\frac{1}{2n}\sum_{0\leq i\leq 2n-1}\sum_{0\leq j\leq 2n-1}h(\Cay(\ZZ_{2n},\{\pm1\});i,j) \\
&=\frac{1}{2n}\left( n\sum_{\ell=1}^{2n-1}h(\Cay(\ZZ_{2n},\{\pm1\});0,\ell)
+n\sum_{\ell=1}^{2n-1}h(\Cay(\ZZ_{2n},\{\pm1\});1,\ell+1) \right) \\
&=\frac{1}{2}( \sum_{\ell=1}^{n-1}h(\Cay(\ZZ_{2n},\{\pm1\});0,2\ell)
+\sum_{\ell=0}^{n-1}h(\Cay(\ZZ_{2n},\{\pm1\});0,2\ell+1) \\
&~+\sum_{\ell=1}^nh(\Cay(\ZZ_{2n},\{\pm1\});1,2\ell)
+\sum_{\ell=1}^{n-1}h(\Cay(\ZZ_{2n},\{\pm1\});1,2\ell+1) ) \\
&=\frac{1}{2p(1-p)}( \sum_{\ell=1}^{n-1}2\ell(2n-2\ell) \\
&~+\sum_{\ell=0}^{n-1}(2\ell(2n-(2\ell+1)+1)+4(1-p)(n-(2\ell+1)+p)) \\
&~+\sum_{\ell=1}^n((2\ell-2)(2n-(2\ell-1)+1)+4p(n-(2\ell-1)+1-p)) \\
&~+\sum_{\ell=1}^{n-1}2\ell(2n-2\ell) ) \\
&=\frac{1}{8p(1-p)}( 8\sum_{\ell=1}^{n-1}\ell(n-\ell)
+\sum_{\ell=0}^{n-1}(4\ell(n-\ell-2(1-p))+4(1-p)(n-1+p)) \\
&~+\sum_{\ell=1}^n(4\ell(n-\ell+2(1-p))-4(n+1)+4p(n+2-p)) \\
&=\frac{1}{2p(1-p)}( 4\sum_{\ell=1}^{n-1}\ell(n-\ell)-\sum_{\ell=1}^n(2p^2-4p+2)+2n(1-p)) \\
&=\frac{1}{p(1-p)}( n^2(n-1)-\frac{n(n-1)(2n-1)}{3}-n(p^2-2p+1)+n(1-p)) \\
&=\frac{n(n-1)(n+1)+3np(1-p)}{3p(1-p)}. \\ 
\end{align*}
\end{proof}

\begin{rem}
For an unweighted graph, the transition probabilities to the adjacent vertices are equal, which corresponds to setting $p=\frac{1}{2}$ in our weighted model.
However, we must account for the difference in the scaling of the graph weights.
In an unweighted graph, the total sum of the weights is $4n$, whereas our weighted model is normalized such that the total sum of the transition probabilities (weights) is $2n$.
By adjusting our formula evaluated at $p=\frac{1}{2}$ with a factor of $\frac{1}{2}$ to correct this scaling difference, it coincides exactly with the unweighted Kirchhoff index of the cycle graph~$C_{2n}$.
\end{rem}

\begin{thm}
We define a weight $w((v_i,v_j))$ on the weighted Cayley graph $\Cay(\ZZ_N,\{+1,+2\})$ as follows: 
\[
w((v_i,v_j))=
\begin{cases}
p &\text{if $j-i=1$,}\\
q &\text{if $j-i=2$,}\\
0 &\text{otherwise,}
\end{cases}
\]
where $i,j\in \ZZ$ and $v_i=i+N\ZZ\in\ZZ_N$.

Then, 
\begin{align*}
&\Kf(\Cay(\ZZ_N,\{+1,+2\}))= \\
&~\frac{N\left[ (p-1)^N(3p-4-N(p-2))-(N-1)(p-2)(2p-3)-3p+4 \right]}{2(p-2)^2((p-1)^N-1)}.
\end{align*}
\end{thm}

\begin{proof}

\begin{align*}
&\Kf(\Cay(\ZZ_N,\{+1,+2\})) \\
&=\sum_{0\leq i<j\leq N-1}r(\Cay(\ZZ_N,\{+1,+2\});v_i,v_j) \\
&=\sum_{0\leq i<j\leq N-1}\frac{h(\Cay(\ZZ_N,\{+1,+2\});i,j)+h(\Cay(\ZZ_N,\{+1,+2\});j,i)}{\sum_{0\leq k,\ell\leq N-1}w((v_k,v_\ell))} \\
&=\sum_{0\leq i<j\leq N-1}\frac{h(\Cay(\ZZ_N,\{+1,+2\});i,j)+h(\Cay(\ZZ_N,\{+1,+2\});j,i)}{N} \\
&=\frac{1}{N}( \sum_{\ell=1}^{N-1}(N-\ell)h(\Cay(\ZZ_N,\{+1,+2\});0,\ell) \\
&~+\sum_{\ell=1}^{N-1}\ell h(\Cay(\ZZ_N,\{+1,+2\});0,\ell) ) \\
&=\sum_{\ell=1}^{N-1}h(\Cay(\ZZ_N,\{+1,+2\});0,\ell) \\
&=\sum_{\ell=1}^{N-1}\frac{N(p-1)((p-1)^\ell-1)-\ell((p-1)^N-1)}{(p-2)((p-1)^N-1)} \\
&=\frac{1}{(p-2)((p-1)^N-1)} \times \\
&~\left( N(p-1)(\sum_{\ell=1}^{N-1}(p-1)^\ell-\sum_{\ell=1}^{N-1}1)-((p-1)^N-1)\sum_{\ell=1}^{N-1}\ell \right) \\
&=\frac{1}{(p-2)((p-1)^N-1)} \times \\
&~\left( N(p-1)(\frac{(p-1)(1-(p-1)^{N-1})}{1-(p-1)}-N+1)-((p-1)^N-1)\frac{N(N-1)}{2} \right) \\
&=\frac{N}{(p-2)((p-1)^N-1)} \times \\
&~\frac{2(p-1)^2-2(N-1)(2-p)(p-1)+(N-1)(2-p)+(p-1)^N(N(p-2)-3p+4)}{2(2-p)} \\
&=\frac{N\left[ (p-1)^N(3p-4-N(p-2))-(N-1)(p-2)(2p-3)-3p+4 \right]}{2(p-2)^2((p-1)^N-1)}.
\end{align*}

\end{proof}

\begin{rem}
For an unweighted graph, the transition probabilities to the adjacent vertices are equal, which corresponds to setting $p=\frac{1}{2}$ in our weighted model. 
We must account for the difference in the scaling of the graph weights.
In an unweighted graph, the total sum of the weights is $2N$, whereas our weighted model is normalized such that the total sum of the transition probabilities (weights) is $N$.
By adjusting our formula evaluated at $p=\frac{1}{2}$ with a factor of $\frac{1}{2}$ to correct this scaling difference, we find that it coincides exactly with the unweighted Kirchhoff index.
\end{rem}

Using the same approach, we can also obtain the weighted Kirchhoff index~\cite{MMS} and the degree Kirchhoff index~\cite{CZ, ZW}.
The resulting formulas coincide exactly with those presented above.

\section*{Acknowledgements}

The authors would like to thank Professor Tsuyoshi Miezaki for his helpful discussions and comments regarding this research. 

This work was supported by JSPS KAKENHI Grant Numbers JP23KJ2020 and JP25K23339, and a Waseda Research Institute for Science and Engineering Grant-in-Aid for Young Scientists (Early Bird). 


\section*{Declarations}

\subsection*{Conflict of interest}

The authors have no relevant financial or non-financial interests to disclose.


\end{document}